\documentclass[12pt]{article}
\usepackage{amsmath, amsfonts, amsthm, amssymb, color, hyperref, extarrows, enumitem, nicefrac,blindtext, mathtools,}

\newcommand{\BigO}[1]{\ensuremath{\operatorname{O}\bigl(#1\bigr)}}

\newtheorem{theorem}{Theorem}[section]

\newtheorem{cor}[theorem]{Corollary}
\newtheorem{lem}[theorem]{Lemma}
\newtheorem{pro}[theorem]{Proposition}

\numberwithin{equation}{section}

\newtheorem{remark}[theorem]{Remark}

\newtheorem{example}{Example}

\usepackage[hyperpageref]{backref}

\usepackage{cite}

\hypersetup{
	colorlinks   = true,
	citecolor    = magenta}
	
\begin{document}
\title{\vspace{-1cm} \bf On solutions to $-\Delta u = V u$ near infinity \rm}
\author{ Yifei Pan \ \  and \ \ Yuan Zhang
}
\date{}

\maketitle

\begin{abstract}
We investigate the unique continuation  property and the sign changing behavior of weak solutions to $-\Delta u =Vu$ near infinity under certain conditions on the  blow-up rate of the potential $V$ near infinity.

\end{abstract}

\renewcommand{\thefootnote}{\fnsymbol{footnote}}
\footnotetext{\hspace*{-7mm}
\begin{tabular}{@{}r@{}p{16.5cm}@{}}
& 2020 Mathematics Subject Classification. Primary 35A02; Secondary 35B99, 32W05. \\
& Key words and phrases.
unique continuation property, Kelvin transform, 
\end{tabular}}

\section{Introduction}
Given  $V\in L^\infty_{loc}(\mathbb R^n), n\ge 2$, 
we study properties of real-valued  weak solutions to
$$    -\Delta u = V u \ \ \text{on}\ \ \mathbb R^n $$  near infinity.  
Landis \cite{La} conjectured that if $V\in L^\infty(\mathbb R^n)$, then every bounded weak solution $u$ such that $|u|\le  \BigO{e^{-k|x|^{1+\epsilon}}}$  near infinity for some $k>0, \epsilon>0$ must be identically $0$. In  \cite{Me},   Meshkov  constructed a nontrivial function $u_0$ which satisfies $ |u_0|\le \BigO{e^{-k|x|^{\frac{4}{3}}}}$ for some $k>0$ and  solves $  -\Delta u = V u $   
for  some  complex-valued $V\in L^\infty(\mathbb R^2)$, thus disproving the Landis conjecture for the complex case.  
 The conjecture has  attracted much attention, particularly since the fundamental works of Bourgain-Kenig \cite{BK} and Kenig  \cite{Ke} in the  real-valued case. See, for instance,   \cite{ABG,  Me, Da, DKW, KSW, Ro} and the references therein.  Recently, Logunov, Malinnikova, Nadirashvili and  Nazarov proved 
the full Landis conjecture when $n = 2$ in \cite{LMNN}.

    Motivated by the Landis conjecture, we first investigate the following unique continuation property at infinity.  Denote by   $(r, \theta)$ the polar coordinates, $ r>0, \theta\in S^{n-1}$. 

\begin{theorem}\label{maini}
      Let $V\in L_{loc}^{\infty}(\mathbb R^n)$. Assume further that   when $|x|>>1$,  $V$ is locally Lipschitz,     \begin{equation}\label{e15}
          V = \BigO{|x|^M} 
      \end{equation}  for some constant  $M\in \mathbb R$, and    \begin{equation}\label{e14}
          \frac{\partial}{\partial r} (r^2V) \ge 0 
      \end{equation} in terms of the polar coordinates $(r, \theta)$.  Let $u\in L^2_{loc}(\mathbb R^n)$ be a weak solution to 
    \begin{equation*}
        -\Delta u = V u\ \ \text{on}\ \ \mathbb R^n.\ \  
    \end{equation*}
 If   $\lim_{r\rightarrow \infty}r^m\int_{|x|>r} |u|^2 = 0$ for each $m\ge 0$, then $u\equiv 0$. \end{theorem}

\begin{remark}
    Although Theorem \ref{maini} is stated for solutions to $-\Delta u= V u$ on the entire Euclidean space $\mathbb R^n$, our proof only requires properties near infinity. This applies similarly to the rest of the results in the paper. Therefore, all of our theorems remain valid with the domain $\mathbb R^n$ replaced by any exterior domain $\mathbb R^n\setminus K$ for some compact subset $K$.
\end{remark}

The $M=0$ case in \eqref{e15} corresponds to Landis's case. In comparison to the conjecture, while imposing an additional condition \eqref{e14}, we  allow $V$ to blow up in a finite order at infinity in \eqref{e15}. In conclusion, we obtain the unique continuation property  by assuming that $u$ vanishes to infinite order in the $L^2$ sense (see Section 2 for the definition) rather than to some negative exponential order at infinity. Along these lines, there has been fruitful investigation  in the literature, though under the stronger superexponential  decay rate assumption. For instance,   Kenig, Silvestre, and Wang \cite{KSW}, and Davey, Kenig, and  Wang \cite{DKW} gave quantitative exponential decay lower bounds for solutions   near infinity when $n=2$ and $M=0$. 
On the other hand, the works of Meshkov \cite{Me},  Froese, Herbst,  Hoffmann-Ostenhof and Hoffmann-Ostenhof \cite{FHHH}, as well as Cruz-Sampedro \cite{Cr}, demonstrate the unique continuation for $-\Delta u = (V+\lambda) u $ with $\lambda\in \mathbb R$, by requiring $ V = \BigO{|x|^{M}}$   and $ e^{\gamma|x|^{1+ \frac{1}{3}\max\{0, 1+2M\}}}u\in L^2(\mathbb R^n)$ for all $\gamma>0$. 
Recently, Davey \cite{Da2}  improved this result by obtaining an  exponential decay lower bound   $  |u(x)|\ge e^{-k|x|^{1+\frac{M}{2}}(\ln|x|)^\frac{3}{2}} $ for some $k>0$ when $n=2$.     In the case when $V\in L^p(\mathbb R^n)$ for some large $p>1$, Davey and Zhu \cite{DZ, DZ2} also provided a similar  exponential decay lower bound $ |u(x)|\ge e^{-k|x|^\tau\ln|x|}$ for some constants $k, \tau>0$ dependent on $p$.

Although   every radial potential satisfying \eqref{e14}  can not change sign near infinity, there are many   non-radial potentials that may change sign near infinity  in the context of Theorem \ref{maini}. See Remark \ref{re3}. We note that in the case when $V\le 0$ on the entire space $\mathbb R^n$, the unique continuation property is trivial for global solutions,  due to  the strong maximum principle. We refer the interested reader  to a result by Davey \cite{Da} for more discussions on bounded sign-changing potentials in the plane. On the other hand,  among other results, Arapostathis, Biswas, and Ganguly \cite{ABG} and  Rossi \cite{Ro} recently proved the strong Landis conjecture with $\epsilon=0$  for nonnegative solutions. Rossi   \cite{Ro} also  proved the conjecture for radial   potentials. 


Regarding the unique continuation property, there are   generally two methods to analyze the maximal order vanishing of the solutions. One method establishes some Carleman-type weighted integral inequalities, see \cite{Ho3, JK, Ke2, KT, Wo, DZ2}, etc. The other method introduced by Garofalo and Lin utilizes monotonicity of Almgren-type frequency functions  in conjunction with doubling inequalities. For instance,   \cite{GL, GL2, Ku, Zh}, etc. Our main tool of the proof  involves a local unique continuation property presented by Agman and Nirenberg \cite{AN}, as well as Li and Nirenberg \cite{LN}, coupled with the application of the Kelvin transform which brings the equation from infinity  to a punctured neighborhood of the origin.   As the Kelvin transform introduces extra strong singularity at the origin, we will need the assumption \eqref{e15}  to resolve the singularity of the new equation at the origin  (in Corollary \ref{HPC}). This assumption also reduces the regularity assumption of the solution to being merely in $L_{loc}^2(B_1\setminus \{0\})$ in Theorem \ref{main}. In fact,  as demonstrated in Example \ref{ex1} and Example \ref{ex2}, the unique continuation property fails at infinity in general if   \eqref{e15} and/or \eqref{e14} is dropped.


One may also compare    Theorem \ref{maini} with another unique continuation result in \cite[Theorem 14.7.2]{Ho}   where $V =  V_1+ \lambda $ with $\lambda >0$ and $|V_1|\le Cr^{-1}$ for some constant $C>0$, under a stronger assumption where both $u$ and $\nabla u$ vanish to infinite order at infinity in the $L^2$ sense. In fact, as a side product, one can immediately obtain 

\begin{cor}\label{ch}
    Given $V$ satisfying all the assumptions as in Theorem \ref{maini}, let $u\in L^2_{loc}(\mathbb R^n)$ be a weak solution to 
    \begin{equation*}
        -\Delta u -Vu=\lambda u\ \ \text{on}\ \ \mathbb R^n\ \  
    \end{equation*}
 for some constant $\lambda\ge 0$. If   $\lim_{r\rightarrow \infty}r^m\int_{|x|>r} |u|^2 = 0$ for each $m\ge 0$, then $u\equiv 0$. 
\end{cor}

In particular, Theorem \ref{maini} can be applied to infer   the  unique continuation property  at infinity when the potential is homogeneous as follows.  See also Example \ref{ex7}  for the construction of more potentials that satisfy the assumptions of  Theorem \ref{maini}.

\begin{cor}\label{ho}
    Let $V\in L^\infty_{loc}(\mathbb R^n\setminus\{0\})$   be a   homogeneous function of order $\alpha$ on $\mathbb R^n$.  Assume either $\alpha=2$, or $\alpha> -2$ and $V\ge 0$ on $\mathbb R^n$, or $\alpha< -2$ and $V\le 0$ on $\mathbb R^n$.   Let $u\in L^2_{loc}(\mathbb R^n)$ be a weak solution to 
    \begin{equation*}
        -\Delta u = V u\ \ \text{on}\ \ \mathbb R^n.
    \end{equation*}    
   If  $\lim_{r\rightarrow \infty}r^m\int_{|x|>r} |u|^2 = 0$ for each $m\ge 0$,  then $u\equiv 0$.
\end{cor}


\begin{cor}\label{ei2}
 Let $u\in L_{loc}^2(\mathbb R^n)$ be an   eigenfunction of $-\Delta$ on $\mathbb R^n$ with  
    $$- \Delta u =  c^2 u \ \ \text{on} \ \  \mathbb R^n$$
    for some constant $c$. Then $u$ does not vanish to infinite order in the $L^2$ sense at $\infty$. Namely, there exists some  $N\ge 0$ such that $\varlimsup_{r\rightarrow \infty}r^N\int_{|x|>r} |u|^2 = \infty$.  \end{cor}


Next, we  study the sign changing behavior of bounded weak solutions near infinity. The following theorem states that when $V$ satisfies certain integrability condition near infinity, then bounded weak solutions must change sign constantly near infinity. 

\begin{theorem}\label{main2}
    Let $V\in L^\infty_{loc}(\mathbb R^n)$. Assume further that    $V\ge 0$ when $|x|>>1$, $\int_{|x|>r} \frac{1}{|x|^{n+2}V}<\infty$ and  $\int_{|x|>r} |x|^{4-2n}V =\infty$ for some $r>>1$. Then every  non-constant bounded weak solution to 
  $$   -\Delta u = V u\ \ \text{on}\ \ \mathbb R^n.$$
    must change sign near $\infty$. Namely, if $u\ge 0 $ (or $u\le 0)$ when $|x|>>1$, then $u\equiv 0$.
\end{theorem}

As indicated by the following corollary, there exist  ample examples for which the assumptions of the potential in Theorem \ref{main2} are satisfied. 

\begin{cor}\label{ho2}
    Let $V\in L^\infty_{loc}(\mathbb R^n\setminus\{0\})$   be a positive  homogeneous function of order $\alpha$    on $\mathbb R^n\setminus\{0\}$.    Let $u $ be a bounded weak solution to 
    \begin{equation*}
        -\Delta u = V u\ \ \text{on}\ \ \mathbb R^n.
    \end{equation*}    
   If    $\alpha>-2$ when $n\le 2$, or $\alpha \ge  n-4  $ when $n\ge 3$, then   $u  $ must change sign near $\infty$.
\end{cor}

According to Liouville's theorem, every  bounded eigenfunction of Laplacian with respect to eigenvalue $0$ (namely,  harmonic function) must be a constant. The following theorem shows that bounded eigenfunctions  with respect to nonzero eigenvalues will constantly change sign near infinity. As seen   in the statement of Corollary \ref{ho2}, the order $\alpha =0$ case is covered by this corollary only when $n\le 4$. The remaining cases when $n\ge 5$  has to be  dealt  with through a different approach. See Section 4 for its proof. 

\begin{theorem}\label{ei}
    Every bounded   eigenfunction  of $-\Delta$ with 
    $$ -\Delta u = c^2 u \ \ \text{on} \ \  \mathbb R^n$$
for a constant $ c\ne 0$ must change sign near $\infty$. 
\end{theorem}

Equivalently, the above theorem states that every bounded nonconstant eigenfunction is neither subharmonic nor superharmonic near infinity. Theorem \ref{ei} (and Theorem \ref{ei2}) fails in general for   eigenfunctions with respect to negative eigenvalues, as indicated in the following example. 

\begin{example}
    Let   $u_0(z)$ be a holomorphic solution on $\mathbb C$ to the following modified Bessel equation 
$$ v''(z) +\frac{1}{z}v'(z)-v(z)=0 \ \ \text{on}\ \ \mathbb C  $$ such that its asymptotic behavior for $|\arg z|<\frac{\pi}{2}$ near $\infty$ is $u_0(z)\sim \frac{1}{\sqrt {2\pi}}z^{-\frac{1}{2}}e^{-z}$. See, for instance, \cite[pp. 183]{Le}. Then $u_0(r)$ is a radial solution to 
$$ -\Delta u = - u\ \ \text{on}\ \ \mathbb R^2.$$
Clearly, $u_0$ is a bounded eigenfunction with respect to eigenvalue $-1$, and it does not change sign near $\infty$.  
\end{example}


 \medskip
 
\noindent{\bf Acknowledgement:}    The authors would like to thank the referees for very helpful suggestions.


\section{Preliminaries}

As mentioned in the introduction, when pulling  the equation from infinity  to the origin  via the Kelvin transform, the procedure results in a new equation that carries singularity at the origin. The purpose of this section is to  extend the new equation across the origin, and meanwhile resolve the singularity of weak solutions there. We shall also discuss  properties of flat functions.     

 Denote by $B_r$ the  ball of radius $r$ centered at $0$ in $\mathbb R^n$.  
 Given $u $ on $\mathbb R^n\setminus B_1$, recall that the Kelvin transform of $u$ is defined by $ w(x): =\frac{1}{|x|^{n-2}}u\left(\frac{x}{|x|^2}\right)$ on $B_1\setminus \{0\}$. One of the properties for the Kelvin transform is that in the  distribution sense,
 \begin{equation}\label{ll}
     \Delta w(x) = \frac{1}{|x|^{n+2}}\Delta u\left(\frac{x}{|x|^2}\right)\ \ \text{on}\ \ B_1\setminus \{0\}. 
 \end{equation} 
 In particular, the harmonicity is invariant under the Kelvin transform, a property that we will apply repeatedly.  

\begin{lem}\label{be}
    Let $V\in L_{loc}^{\infty}(\mathbb R^n\setminus B_1)$, and $u\in L^2_{loc}(\mathbb R^n\setminus B_1)$.  Define  \begin{equation}\label{e16}
        w(x): = \frac{1}{|x|^{n-2}}u\left(\frac{x}{|x|^2}\right)\ \ \text{ and} \ \ W(x): = \frac{1}{|x|^{4}}V\left(\frac{x}{|x|^2}\right)\ \ \text{on}\ \ B_1\setminus \{0\}. 
    \end{equation}
    Then the following properties hold. \\
  (1)   $W\in L_{loc}^{\infty}(B_1\setminus \{0\})$, and $w\in L_{loc}^2(B_1\setminus \{0\})$.\\
   (2)  If    $  V = \BigO{ |x|^M} $
     near $\infty$   for some constant $M$, then $W = \BigO{\frac{1}{|x|^{M+4}}}$ near $0$.\\
  (3)  For every $r>0$, one has $$\int_{|x|>\frac{1}{r}} \frac{1}{|x|^{n+2}V} = \int_{|x|<r} \frac{1}{|x|^{n+2}W } \ \ \text{and}\ \ \int_{|x|> \frac{1}{r} } |x|^{4-2n}V = \int_{|x|<r} W.$$
\\
    (4)  If      $$\lim_{r\rightarrow \infty}r^m\int_{|x|>r} |u|^2 = 0$$ for some constant $m\ge 0$, then $w\in L_{loc}^2(B_1)$ with
   $$\lim_{r\rightarrow 0} r^{-(m+4)}\int_{|x|<r} |w|^2=0.$$  \\
    (5)  If $u$ is a weak solution to 
    \begin{equation*}
        -\Delta u = V u\ \ \text{on}\ \ \mathbb R^n\setminus B_1,\ \  
    \end{equation*}
then $w$ is a weak solution to
     \begin{equation*}
    \begin{split}
       - \Delta w  = W  w\ \ \text{on}\ \ B_1\setminus \{0\}.
    \end{split}
\end{equation*}

\end{lem}

\begin{proof}
{(1)} and  {(2)} are true by definition. For (3), we apply change of coordinates $x\rightarrow \frac{x}{|x|^2}$  to verify the equalities directly as follows.
\begin{equation*}
    \begin{split}
     & \int_{|x|<r} \frac{1}{|x|^{n+2}W }  =  \int_{|x|<r} \frac{1}{|x|^{n-2}V\left(\frac{x}{|x|^2}\right) }  =  \int_{|x|>\frac{1}{r}} \frac{|x|^{n-2}}{V(x) }|x|^{-2n} =   \int_{|x|>\frac{1}{r}} \frac{1}{|x|^{n+2}V};  \\
     & \int_{|x|<r} W  =  \int_{|x|<r} \frac{1}{|x|^{4}}V\left(\frac{x}{|x|^2}\right)   =  \int_{|x|>\frac{1}{r}} |x|^{4}V(x) |x|^{-2n} = \int_{|x|> \frac{1}{r} } |x|^{4-2n}V. 
    \end{split}
\end{equation*}
Here we  used the fact that  the Jacobian of the coordinates change  is    $|x|^{-2n}$.

For the decay property of $w$ in {(4)}, by the same change of coordinates as above,  
\begin{equation*}
    \begin{split}
      \int_{|x|<r} |w|^2 =&\int_{|x|<r} \frac{1}{|x|^{2n-4}}\left|u\left(\frac{x}{|x|^2}\right)\right|^2  =  \int_{|x|>\frac{1}{r}} |x|^{2n-4}\left|u(x)\right|^2 |x|^{-2n}\\
      =&\int_{|x|>\frac{1}{r}} |x|^{-4}\left|u(x)\right|^2 \le r^{4} \int_{|x|>\frac{1}{r}}\left|u\right|^2 .  
    \end{split}
\end{equation*} 
Note that by assumption, $\lim_{r\rightarrow 0}r^{-m}\int_{|x|>\frac{1}{r}} |u|^2 =\lim_{r\rightarrow \infty}r^m\int_{|x|>r} |u|^2 = 0$. Hence  
$$ \lim_{r\rightarrow 0} r^{-(m+4)}\int_{|x|<r} |w|^2   \le \lim_{r\rightarrow 0} r^{-m-4} r^4\int_{|x|>\frac{1}{r}}\left|u\right|^2 =0.$$

 {(5)} follows from \eqref{ll} by a straightforward computation: on $  B_1\setminus \{0\}$,
 \begin{equation*}
    \begin{split}
       - \Delta w(x) = -\frac{1}{|x|^{n+2}}\Delta u\left(\frac{x}{|x|^2}\right) = \frac{1}{|x|^{n+2}}V\left(\frac{x}{|x|^2}\right)u\left(\frac{x}{|x|^2}\right) = \frac{1}{|x|^{4}}V\left(\frac{x }{|x|^2}\right)w(x)= W(x) w(x)
    \end{split}
\end{equation*}
in the distribution sense.

\end{proof}

Next we    extend the weak solutions across the singularity at the origin, with the help of a Harvey-Polking type (see \cite{HP})   lemma below for an isolated singularity. Throughout the rest of the paper we use the following notation:  two quantities $A$ and $B$ are said to satisfy $A\lesssim B$, if  $A\le CB$ for some constant $C>0$ that depends only possibly on $n$.  
 
\begin{lem}\label{HP}
  Let $P(x,D)$ be a linear differential operator of order $l$ and $f\in L^1(B_1)$. If $u\in L^1(B_1)$ is a weak solution to $P(x, D) u = f$ on $B_1\setminus \{0\}$  and \begin{equation}\label{fv}
    \lim_{r\rightarrow 0} r^{-l}\int_{|x|<r}|u |    =0, 
\end{equation}  then $u$ is a weak solution to
$P(x, D) u = f$ on $B_1$.

\end{lem}

\begin{proof}
    Let $\phi^r$ be a smooth function on $B_1$ such that $\phi^r = 1$ on $B_{\frac{r}{2}}$,  $\phi^r = 0$ outside $B_{r}$ and $|\nabla^k \phi^r|\lesssim r^{-k}$ on $B_r$, $k\le l$. Then for any testing function $\phi$ on $B_1$, $(1-\phi^r)\phi$ is a testing function on $B_1 \setminus \{0\} $. Thus 
    $$ \langle P(x,D) u -f , (1-\phi^r)\phi\rangle =0, $$
 and so
    \begin{equation*}
        \begin{split}
            \langle P(x,D) u -f , \phi\rangle =  \langle P(x,D) u -f , \phi^r\phi\rangle  
            =  \langle  u,  {}^tP(x,D)(\phi^r\phi)\rangle -  \langle  f , \phi^r\phi\rangle.
        \end{split}
    \end{equation*}
    Passing $r$ to $0$, since $f\in L^1(B_1)$, $$ \langle  f , \phi^r\phi\rangle \lesssim \int_{B_{r} }|f| \rightarrow 0.$$
  On the other hand, by assumption \eqref{fv}, $$\langle  u,  {}^tP(x,D)(\phi^r\phi)\rangle \lesssim r^{-l} \int_{|x|<r }|u|
    \rightarrow 0.$$   
  We thus have the desired identity $  \langle P(x,D) u -f , \phi\rangle =0.  $ 
  
\end{proof}

\begin{lem}\label{HP2}
    Let  $W\in L^{\infty}_{loc}(B_1\setminus\{0\})$, $W =\BigO{\frac{1}{|x|^M}}$  near $0$ for some constant $M\ge 0$, and  $w\in L^2_{loc}(B_1)$ be a weak solution to 
    \begin{equation*}
       - \Delta w = W w\ \ \text{on}\ \ B_1\setminus \{0\}.
    \end{equation*}
If   further \begin{equation}\label{e19}
   \lim_{r\rightarrow 0} r^{-2M-3}\int_{|x|<r} |w|^2=0,
\end{equation}  
  then $w\in W_{loc}^{2,2}(B_1)\cap W_{loc}^{2,p}(B_1\setminus \{0\})$ for all $p< \infty$, and $w$ is a weak solution to  $$-\Delta w = W w \ \ \text{on}\ \   B_1.$$ 
\end{lem}

\begin{proof}
First,   for all  $r<<1$,   by \eqref{e19}
 \begin{equation}\label{e20}
     \int_{|x|<r}|Ww|^2 \lesssim  \int_{|x|<r}\frac{|  w|^2}{|x|^{2M}} = \sum_{j= 0}^\infty\int_{2^{-j-1}r<|x|<2^{-j}r}    \frac{|  w|^2}{|x|^{2M}} \le   \sum_{j= 0}^\infty 2^{2(j+1)M}r ^{-2M}\int_{ |x|<2^{-j}r}     |  w|^2\le C r^3  
 \end{equation}
 for some constant $C$ dependent only on $M$.    So $Ww\in L_{loc}^2(B_1)$.
 On the other hand, by H\"older inequality and    \eqref{e19} again,  $$ \lim_{r\rightarrow 0} r^{-2}\int_{|x|<r}|w |  \le  \lim_{r\rightarrow 0} r^{ -2 +\frac{n}{2}}\left(\int_{|x|<r} |w|^2\right)^\frac{1}{2}  \le \lim_{r\rightarrow 0}  \left(r^{-3}\int_{|x|<r} |w|^2\right)^\frac{1}{2}  =0.$$
   As a consequence of the Harvey-Polking type  Lemma \ref{HP}, $-\Delta w = W w $ holds true on $B_1$ in  the distribution sense. PDE theory then tells that $w\in W_{loc}^{2,2}(B_1)$. In particular, $w\in L_{loc}^{p_0}(B_1)$ where $p_0 = \frac{2n}{n-4}$ if $n>4$, or any number less than $\infty$ if $n\le 4$ by the Sobolev embedding theorem. 
 
  Restricted on $B_1\setminus \{0\}$, since  $Ww\in L_{loc}^{p_0}(B_1\setminus \{0\})$, we further have $w\in W^{2, p_0}_{loc}(B_1\setminus\{0\}). $ With a boot-strap argument, we eventually obtain $w\in   W_{loc}^{2,p}(B_1\setminus \{0\})$ for all $p<\infty$.

 
\end{proof}

\begin{cor}\label{HPC}
      Let  $W\in L^{\infty}_{loc}(B_1\setminus\{0\})$, $W =\BigO{\frac{1}{|x|^M}}$  near $0$ for some constant $M$, and  $w\in L^2_{loc}(B_1)$ be a weak solution to 
    \begin{equation*}
       - \Delta w = W w\ \ \text{on}\ \ B_1\setminus \{0\}, 
    \end{equation*}
If      $w$ vanishes to infinite order in the $L^2$ sense at $0$, then $w\in W_{loc}^{2,2}(B_1)\cap W_{loc}^{2,p}(B_1\setminus \{0\})$ for all $p< \infty$, and $w$ is a weak solution to  $$-\Delta w = W w \ \ \text{on}\ \   B_1.$$ 
\end{cor}

 \begin{proof}
     The $M\ge 0$ case  follows directly from Lemma \ref{HP2}.  For $M<0$, we   can reduce to the $M=0$ case by noting that $W =\BigO{\frac{1}{|x|^M}}  =\BigO{\frac{1}{|x|^0}}$ near $0$. 
     
 \end{proof}

Next, we discuss properties of flat functions at $0$. A function $w\in L^2_{loc}(\mathbb R^n)$ is said to vanish to infinite order (or, flat) in the $L^2 $ sense at $x_0\in \mathbb R^n$  if for every $m\ge 0$,
 \begin{equation}\label{e25}
      \lim_{r\rightarrow 0} r^{-m}\int_{|x-x_0|<r}|w |^2    =0.
 \end{equation}
In view of (the proof to) Lemma \ref{be} {(4)}, it makes sense to call that  a function $u\in L^2_{loc}(\mathbb R^n)$ vanishes to infinite order in the $L^2 $ sense at infinity, if for   every $m\ge 0$, 
 $$\lim_{r\rightarrow \infty}r^m\int_{|x|>r} |u|^2 = 0. $$
 Thus, as a consequence of Lemma \ref{be} {(4)}, if $u$ vanishes to infinite order in the $L^2 $ sense at infinity, then $w$ defined in \eqref{e16} vanishes to infinite order in the $L^2 $ sense at $0$.

   In general the   flatness of a function does not necessarily imply the flatness of its derivatives. For instance, $w (x):  = e^{-\frac{1}{x^2}}\cos (e^{\frac{1}{x^2}}), x\in \mathbb R$  is flat at $0$. However,  the derivative $w'$ is not even $L^1_{loc}$ at $0$. On the other hand,  the following lemma states that the flatness naturally passes onto its derivatives if $w$ is a weak solution to $-\Delta w = W w $ when $W$ blows up at most in a finite order at $0$.  We shall need the following well-known inequality  with proof provided below for completeness.

\begin{lem}
    Let $f\in W_{loc}^{2,2}(B_1)$. Then for any $0<r<\frac{1}{2}$,
\begin{equation}\label{in}
    \int_{|x|<r}|\nabla f|^2\lesssim \frac{1}{r^2}\int_{|x|<2r}|f|^2 + r^2 \int_{|x|<2r} |\Delta f|^2.
\end{equation}
\end{lem}

\begin{proof}
    Let $\phi$ be a nonnegative smooth function with compact support  on $B_{2r}$ such that $\phi =1$ on $B_r$, $\phi\le 1$ and $|\nabla^k \phi|\lesssim \frac{1}{r^k}, k\le 2$ on $B_{2r}$. Since $ \phi \nabla f= \nabla (\phi f) - f\nabla \phi, $
 we have 
\begin{equation*}
    \begin{split}
      \int_{|x|<2r} \phi^2 |\nabla  f|^2  \le 2\int_{|x|<2r} |f \nabla \phi |^2 + 2\int_{|x|<2r}| \nabla (\phi f)|^2.  
    \end{split}
\end{equation*}
Making use of Stokes' theorem to the second term on the right,  one further has
\begin{equation}\label{h2}
    \int_{|x|<2r} \phi^2 |\nabla  f|^2  \le 2\int_{|x|<2r} |f \nabla \phi |^2 + 2\int_{|x|<2r}\phi |f| |\triangle (\phi f)|.
\end{equation}

On the other hand, note that $ \triangle (\phi f) = f\triangle \phi + \phi\triangle f +2\nabla \phi\cdot \nabla f $. 
By the choice of $\phi$ and with a repeated application of the Schwartz inequality,
\begin{equation}\label{h1}
\begin{split}
     \int_{|x|<2r} \phi |f|| \triangle (\phi f) |\le   & \int_{|x|<2r} \phi |f|^2 | \triangle \phi|+  \int_{|x|<2r} \phi^2 |f| |\triangle  f| +  2\int_{|x|<2r} \phi |f||\nabla\phi ||\nabla f| \\
     \lesssim    & \frac{1}{r^2}\int_{|x|<2r} |f|^2+  \int_{|x|<2r}   \frac{1}{r}|f|\cdot r|\triangle  f| + \int_{|x|<2r} \phi |\nabla f|\cdot   \frac{1}{r}|f|  \\
 \lesssim & \frac{1}{r^2}\int_{|x|<2r} |f|^2 +  r^2\int_{|x|<2r}   |\triangle  f|^2 +   \frac{1}{4} \int_{|x|<2r}\phi^2|\nabla f|^2.
\end{split}
\end{equation}
Combining \eqref{h2}-\eqref{h1}, we have 
$$\int_{|x|<2r}\phi^2 |\nabla  f|^2\lesssim\frac{1}{r^2}\int_{|x|<2r} |f|^2 +  r^2\int_{|x|<2r}   |\triangle  f|^2,  $$
from which \eqref{in} follows.

\end{proof}

\begin{lem}\label{flatd}
 Suppose $w\in L^2_{loc}(B_1)$ vanishes to infinite order in the $L^2$ sense at $0$. Then the following holds. \\
  (1) If  $f =\BigO{\frac{1}{|x|^M}}  $   near $0$ for some constant $M$, then  $fw\in L^2$ near $0$, and  vanishes to infinite order in the $L^2$ sense at $0$.\\
   (2)  If $w$ is a weak solution to
    \begin{equation*}
        -\Delta w = W w\ \ \text{on}\ \ B_1.
    \end{equation*}
    for some $W\in L^{\infty}_{loc}(B_1\setminus\{0\})$ such that $W = \BigO{\frac{1}{|x|^M}}$  near $0$ for some constant $M$. Then 
   $\nabla w$  vanishes to infinite order in the $L^2$ sense at $0$. In particular,  $\nabla_\theta w$ and $ w_r$ with respect to   the polar coordinates $(r, \theta)$ vanish  to infinite order in the $L^2$ sense at $0$.
\end{lem}

\begin{proof}
{(1)} can be proved by making use of a similar argument as in \eqref{e20}, with the assumption \eqref{e19}  replaced by the $L^2$ flatness \eqref{e25} of $w$ at $0$.

For {(2)}, since $w$ vanishes to infinite order in the $L^2$ sense at $0$,   $w\in W^{2,2}_{loc}(B_1)$ by Corollary \ref{HPC}. On the other hand,  by {(1)}  $Ww$ also vanishes to infinite order in the $L^2$ sense at $0$.  Thus  we apply \eqref{in} to $w$ and  obtain   the  flatness of $\nabla w$ in the $L^2$ sense. 

That   $\nabla_\theta w$ and $ w_r$   vanish  to infinite order in the $L^2$ sense at $0$  follows directly from the identity 
$$ |\nabla w|^2 = |w_r |^2 + \frac{1}{r^2} |\nabla_\theta v|^2 $$
in terms of the polar coordinates   $(r, \theta)$. 

\end{proof}

 \medskip

\begin{lem}\label{seq}
    Suppose that $w\in L^2(B_1)\cap C(B_1\setminus \{0\})$ satisfies 
    $$  \lim_{r\rightarrow 0} r^{-1} \int_{|x|< r} |w|^2=0.$$
    Then there exists $r_j\rightarrow 0$ such that 
    $$ \int_{|x| = r_j}|w|^2\rightarrow 0. \ \ $$
\end{lem}
 
\begin{proof}
First note that by assumption,  $\int_{|x| = r}|w|^2 $ as a function of $r$ is in $C((0,1))$.  By the mean-value theorem, for each $j\ge 1$, there exists $r_j\in (\frac{r}{2^{j}}, \frac{r}{2^{j-1}})$ such that 
$$ \int_{ 2^{-j}r<|x|< 2^{-j+1}r} |w|^2 = \int_{ 2^{-j}r}^{2^{-j+1}r} \int_{|x| = s} |w|^2 =  2^{-j}r  \int_{|x| = r_j} |w|^2.$$
When $j\rightarrow \infty$, we have $r_j\rightarrow 0$ and thus  by assumption,
    \begin{equation*}
        \int_{|x|= r_j}|w|^2\le \frac{2^j}{r}\int_{2^{-j}r<|x|< 2^{-j+1}r} |w|^2 \le  2 \left(2^{-j+1}r\right)^{-1}\int_{ |x|< 2^{-j+1}r} |w|^2  \rightarrow 0.
    \end{equation*}

\end{proof}

\section{Unique continuation property at infinity}
To prove Theorem \ref{maini}, we shall make use of a  unique continuation property near the origin that was established by Li and Nirenberg \cite[Theorem 10]{LN} for smooth solutions without an  assumption on   the potential's growth near the singularity. However, to adapt its proof to our context of $L^2$ weak solutions,  considerable modifications are needed in order for its application. It is also crucial to point out that a specific boundary term  involving $V$ in \cite{LN} can not be discarded,  due to insufficient information on the blow-up rate of  the potential near the origin in \cite{LN}. By imposing  the additional assumption \eqref{e15},   we have the desired control of singularity for the potential near the origin when pulling   to the origin by the Kelvin transform, so that the aforementioned boundary term can be managed.     For clarification  we are compelled to provide the detailed proof below.

\begin{theorem}\label{main}
   Let $W\in  L^{\infty}_{loc}(B_1\setminus\{0\})$ with  $W = \BigO{\frac{1}{|x|^M}}$  near $0$ for some constant $M$. Assume further there exists some $0<r_0<1$ such that  $W$ is locally Lipschitz in $ B_{r_0}\setminus\{0\} $ with  $ \frac{\partial}{\partial r} (r^2W) \le 0 $  on $0<r<r_0$.  Let  $w\in L_{loc}^2(B_1\setminus \{0\})$ be a weak solution to 
    \begin{equation*}
        -\Delta w = W w\ \ \text{on}\ \ B_1\setminus \{0\}.
    \end{equation*}
 If $w$ vanishes to infinite order in the $L^2$ sense at $0$, then $w\equiv 0$. 
\end{theorem}

\begin{proof}  By Corollary \ref{HPC}, we have $w\in W_{loc}^{2,2}(B_1)\cap W_{loc}^{2,p}(B_1\setminus \{0\})$ for all $p< \infty$, and  $$-\Delta w = W w \ \ \text{on}\ \   B_1.$$ 
In particular, $w\in C^1(B_1\setminus \{0\})$ by Sobolev embedding theorem. Under the polar coordinates $(r, \theta),   -\Delta w = Ww$ is written as 
    \begin{equation*}
        r^2w_{rr} + (n-1) rw_r + \Delta_\theta w = -r^2 Ww, \ \ 0<r<1, \ \theta\in S^{n-1}.
    \end{equation*}
    Setting $r = e^s, s<0$, then  a direct computation gives
    \begin{equation*}
        \begin{split}
            &w_s = w_r e^s = r w_r,\\
        & w_{ss} = w_{rr} e^{2s} + w_r e^s = r^2 w_{rr} + rw_r.
        \end{split}
    \end{equation*}
    Thus we have 
 \begin{equation}\label{ue}
     w_{ss} + (n-2)w_s +\Delta_\theta w = -e^{2s} Ww,  \ \ s<0, \ \theta\in S^{n-1}.
 \end{equation}   

Now let $w = e^{as} v$ with $a = -\frac{n-2}{2}$.  One can further verify that 
\begin{equation*} 
    \begin{split}
&w_s = ae^{as} v + e^{as}v_s = e^{as} (av +v_s),\\
 & w_{ss} =ae^{as}(av+v_s) + e^{as}(av_s+v_{ss})= e^{as}(v_{ss} + 2av_s +a^2v).    
    \end{split}
\end{equation*}
Plugging the above into \eqref{ue}, after simplification we obtain 
\begin{equation}\label{ve}
    v_{ss} +\Delta_\theta v = \left(-e^{2s}W +\frac{(n-2)^2}{2}\right)v: = mv,
\end{equation}
where $m = -e^{2s}W + \frac{(n-2)^2}{2}$.  By assumption, there exists some $s_0<0$ such that \begin{equation}\label{e8}
    m_s = -(e^{2s}W)_s = -(r^2W)_re^s\ge 0, \ \ s<s_0. 
\end{equation}

On the other hand, by the assumption of the $L^2$ flatness of  $w$  at $0$, we have $w_s(=rw_r)$ and $\nabla_\theta w$ vanish to infinite order in the $L^2$ sense  at $r=0$ by Lemma \ref{flatd}. Moreover, since $v =r^{-a}w$, one infers that $v_s = r^{-a}w_s -av $ and $\nabla_\theta v = r^a \nabla_\theta w $. Thus   $v_s$ and $\nabla_\theta v$ vanish to infinite order in the $L^2$ sense  at $r=0$ by Lemma \ref{flatd} again. Similarly, as $m= \BigO{\frac{1}{|x|^{N}}}$ near $r=0$ by definition, where $N:=\max\{M-2, 0\}$, we also have $|m|^\frac{1}{2}v$ vanishes to infinite order in the $L^2$ sense  at $r=0$. Thus applying Lemma \ref{seq} to $|v_s|+|\nabla_\theta v|+  |m^\frac{1}{2}v|$, there exists $s_j\rightarrow -\infty$ such that 
\begin{equation}\label{e4}
    \int_{S^{n-1}}\left.v^2_s\right|_{s=s_j} +   \int_{S^{n-1}}\left.|\nabla_\theta v|^2\right|_{s=s_j} 
 +  \int_{S^{n-1}}\left. |m|v^2\right|_{s=s_j} \rightarrow 0. 
\end{equation}

Multiply both sides of \eqref{ve} by $2v_s$ and integrate from $s_j$ to $s$ about $s$ and over $S^{n-1}$ about $\theta$. Then
\begin{equation}\label{e1}
    \int_{S^{n-1}}\int^s_{s_j} 2v_s v_{ss}   + 2 \int_{S^{n-1}}\int_{s_j}^s \Delta_\theta v v_s  = 2 \int_{S^{n-1}}\int_{s_j}^s mvv_s 
\end{equation}
 For the second term on the left hand, we use  Green's theorem and obtain
\begin{equation*}
    \begin{split}
   2 \int_{S^{n-1}}\int_{s_j}^s \Delta_\theta v v_s &=    - 2\int_{s_j}^s\int_{S^{n-1}} \nabla_\theta v\cdot \nabla_\theta v_s =    - \int_{S^{n-1}}  \int_{s_j}^s(|\nabla_\theta v|^2)_s  \\
    &=   - \int_{S^{n-1}}   |\nabla_\theta v|^2  + \int_{S^{n-1}} \left.   |\nabla_\theta v|^2\right|_{s=s_j} =   \int_{S^{n-1}}   v \Delta_\theta v   + \int_{S^{n-1}}  \left.  |\nabla_\theta v|^2\right|_{s=s_j}.
    \end{split}
\end{equation*} For the rest of the terms in \eqref{e1}, we compute directly to have  $$\int_{S^{n-1}}\int^s_{s_j} 2v_s v_{ss}    = \int_{S^{n-1}}\int^s_{-\infty} (v^2_s)_s  =  \int_{S^{n-1}}\left.v^2_s - v^2_s\right|_{s=s_j},$$
\begin{equation*}
    \begin{split}
     2 \int_{S^{n-1}}\int_{s_j}^s mvv_s    = \int_{S^{n-1}}\int_{s_j}^s (mv^2)_s - \int_{S^{n-1}}\int_{s_j}^s m_s v^2 =  \int_{S^{n-1}}   mv^2 - \int_{S^{n-1}}\int_{s_j}^s m_s v^2  -  \int_{S^{n-1}}\left. mv^2\right|_{s=s_j}
    \end{split}
\end{equation*}
Altogether we infer
\begin{equation*}
    \begin{split}
      \int_{S^{n-1}}  v_s^2 =   & -\int_{S^{n-1}}   v \Delta_\theta v   + \int_{S^{n-1}}   mv^2 - \int_{S^{n-1}}\int_{s_j}^s m_s v^2 \\
      &+\left(\int_{S^{n-1}}\left. v^2_s\right|_{s=s_j} -   \int_{S^{n-1}}\left.|\nabla_\theta v|^2\right|_{s=s_j} - \int_{S^{n-1}}\left. mv^2\right|_{s=s_j}\right).
    \end{split}
\end{equation*}
Here the last term   $\int_{S^{n-1}}\left. mv^2\right|_{s=s_j}$  could potentially be  troublesome  due to the presence of uncontrolled growth of $W$ in $m$ at $r=0$, but seems to have been overlooked in \cite{LN}. As seen in \eqref{e4}, the additional assumption  $W = \BigO{\frac{1}{r^M}}$ allows us  to eliminate this term as $j\rightarrow \infty$. In detail, letting $s_j\rightarrow -\infty$, and making use of \eqref{e8}-\eqref{e4}, we have 
\begin{equation}\label{e2}
      \int_{S^{n-1}}  v_s^2 \le -\int_{S^{n-1}}   v \Delta_\theta v+ \int_{S^{n-1}}   mv^2.
\end{equation}

Now  we consider 
$$\rho(s) = \int_{S^{n-1}} v^2(s, \theta)d\theta,\ \ s<s_0.  $$
We shall prove that
\begin{equation}\label{le}
     \rho_s^2\le \rho\rho_{ss},\ \ s<s_0.
\end{equation}
In fact, $$\rho_s = 2\int_{S^{n-1}} vv_s,$$ and by \eqref{ve},
\begin{equation*}
    \begin{split}
        \rho_{ss} = 2\int_{S^{n-1}} v_s^2 +vv_{ss} =  2\int_{S^{n-1}} v_s^2 + 2 
 \int_{S^{n-1}}v(-\Delta_\theta v + mv).  \end{split}
\end{equation*}
 Making use of H\"older inequality and \eqref{e2},   we   have
\begin{equation*}
      \begin{split}
     \rho_s^2 \le & 4 \int_{S^{n-1}} v^2\int_{S^{n-1}} v_s^2 = 4\rho\int_{S^{n-1}} v_s^2\\
     \le &  \rho\left(2\int_{S^{n-1}} v_s^2    +2\int_{S^{n-1}}   v(- \Delta_\theta v) +2 \int_{S^{n-1}}   mv^2\right) = \rho\rho_{ss}.
     \end{split}
\end{equation*}
 \eqref{le} is proved. In particular, this implies that whenever $\rho>0$,
 \begin{equation}\label{con}
     (\log \rho)_{ss}\ge 0.
 \end{equation}

Assume by contradiction that  there exists  $\bar s<s_0 $ such that $\rho(\bar s) >0$.  Let $s^\sharp$ be the infimum  of all $\hat s$ such that $\rho>0$   on the interval $(\hat s, \bar s)$. Then by \eqref{con},
\begin{equation}\label{e27}
    \log \rho(s) \ge \log\rho(\bar s) + \frac{d}{ds} \log\rho(\bar s)(s-\bar s), s^\sharp<s<\bar s. 
\end{equation} 
In particular, this implies that $\rho(s^\sharp)>0$ whenever $s^\sharp $ is finite, which would violate the choice of $s^\sharp$.  Thus $s^\sharp=-\infty$. Namely,  $\rho >0$ for all $s<\bar s$.  Consequently, by \eqref{e27} there exists some $C_1, C_2$ such that 
$\rho(s) \ge C_1 e^{C_2s}$ for all $s<\bar s$. Equivalently, for all $r<r_0: = e^{\bar s}$,
$$ \int_{S^{n-1}}v^2 \ge C_1r^{C_2}. $$
Recalling that $w =  r^{-\frac{n-2}{2}}v$, we further have when $r<<1$,
\begin{equation*}
    \begin{split}
     \int_{|x|<r} w^2 =  \int_{0}^rt^{n-1}\int_{S^{n-1}} w^2 = \int_{0}^rt^{n-1}t^{-n+2}\int_{S^{n-1}} v^2\ge C_1\int_{0}^r t^{1+C_2},
    \end{split}
\end{equation*}
which is either infinite or $\BigO{r^{2+C_2}}$. This  contradicts with the  $L^2$ flatness of $w$ at $0$. Thus $\rho = 0$ for all $s<s_0$ and so is $w$ on $|x|< e^{s_0}$. We further apply the  classical unique continuation property for $L^\infty$ potentials (see, for instance, \cite{JK}) to get $w\equiv 0$. 

\end{proof}
\medskip

Theorem \ref{maini} is a direct consequence of the above theorem after imposing the Kelvin transform. 

\begin{proof}[Proof of Theorem \ref{maini}:] Apply the Kelvin transform to obtain $w$ and $W$  under the setting in Lemma \ref{be}. In order to apply Theorem \ref{main}, we only need to verify that $(r^2W)_r\le 0 $ for $0<r<<1$. Clearly $W$ is locally Lipschitz on $B_{r_0}\setminus \{0\} $ for some $r_0>0$. In terms of the polar coordinates $(r, \theta)$,
 \begin{equation*}
     \begin{split}
         (r^2W)_r(r, \theta) =&  \left(\frac{V(\frac{1}{r}, \theta)}{r^2}\right)_r  = -\frac{2}{r^3} V\left(\frac{1}{r}, \theta\right) -\frac{1}{r^4} V_r \left(\frac{1}{r}, \theta\right)\\
         =& -\frac{1}{r^2} \left(\frac{2}{r}V\left(\frac{1}{r}, \theta\right) +\frac{1}{r^2}V_r \left(\frac{1}{r}, \theta\right)\right).
     \end{split}
 \end{equation*}
On the other hand,  when $r<<1$, by assumption $$0 \le  (r^2V)_r\left(\frac{1}{r}, \theta\right) = \left(2rV +r^2V_r \right)\left(\frac{1}{r}, \theta\right)  = \frac{2}{r}V\left(\frac{1}{r}, \theta\right) + \frac{1}{r^2} V_r\left(\frac{1}{r}, \theta\right).$$
So we have $ (r^2W)_r\le 0 $.  Theorem \ref{main} thus applies to give $w =0$ on $B_1$. Then  $u=0$ outside $B_1$. That $u=0$ on $B_1$ is a direct consequence of  the classical  unique continuation property with $L_{loc}^\infty$ potentials. 
 
 \end{proof}
\medskip

\begin{remark}\label{re2}
    An inspection of the proof to Theorem \ref{maini} indicates that the local Lipschitzian assumption of $V$ in Theorem \ref{maini} can be weakened to  merely assuming  $V_r\in  L^\infty_{loc}$  for $r>>1$ in the polar coordinates $(r, \theta)$. This simple observation will be used in some of our later proofs and examples.
\end{remark}

\begin{remark}\label{re3}
In the special case when $V$ is radial near infinity, the assumption \eqref{e14} actually implies that $V$  does not change sign near infinity. In fact,  according to   \eqref{e14}, $V(r_2)\ge \frac{r_1^2}{r_2^2}V(r_1)$ for all $0<r_1<r_2$. Consequently,  either $V(r)<0$ for all $r>>1$, or if $V(r_0)\ge 0$ at some $r_0>0$, then $V(r) \ge 0$ for all $r>r_0$. 

However, for general non-radial potentials, as there is no regularity assumption on $V$ along the  $\theta$ direction as mentioned in Remark \ref{re2}, one can easily construct examples of $V$ such that $V$ has opposite signs  along at least two $\theta$ directions. For instance, for each $(r, \theta)\in \mathbb R^2\setminus B_1$, let $V(r, \theta) = -\frac{1}{r^2}$ if $\theta\in (0, \pi]$, and $ V(r, \theta) = r $ if $\theta\in (\pi, 2\pi]$.  Clearly, this $V$ satisfies \eqref{e14} but changes sign near $\infty$.
\end{remark}
\medskip

\begin{proof}[Proof of Corollary \ref{ch}:] In view of Theorem \ref{maini}, we only need to verify that  \eqref{e15} and \eqref{e14} hold with $V$ replaced by $V+\lambda$ for $\lambda>0$. Indeed, given $V$ satisfying the assumptions in Theorem \ref{maini} and $\lambda>0$, $V+\lambda  = \BigO{|x|^{\max\{M, 0\}} }$ when $|x|>>1$. Moreover, it is straightforward to verify  that  
    $$      \frac{\partial}{\partial r} \left(r^2(V+\lambda) \right) =  \frac{\partial}{\partial r} (r^2V ) + 2\lambda r > 0, \ \ r>> 1.$$ 
    The proof is complete. 
    
\end{proof}

\begin{proof}[Proof of Corollary \ref{ho} and Corollary \ref{ei2}:]  Note that every  homogeneous function  in $L^\infty_{loc}(\mathbb R^n\setminus\{0\})$ is automatically Lipschitz along the $r$ direction. In view of Theorem \ref{maini} and Remark \ref{re2}, we just need to show that $ (r^2 V)_r\ge 0$ for Corollary \ref{ho}.  Since  $ V(r, \theta) = r^\alpha V(1, \theta)$ for $ \theta\in S^{n-1}$, by a straight-forward computation we get 
$$ (r^2 V)_r(r, \theta)= (r^{2+\alpha} V(1, \theta))_r = (2+\alpha)r^{\alpha +1} V(1, \theta) = (2+\alpha)r V(r, \theta).   $$
This is always non-negative in either one of the three cases in the corollary. Corollary \ref{ho} is thus proved. 

Corollary \ref{ei2} is a special case of Corollary \ref{ho} with $\alpha=0$ and $V=c^2\ge 0$. Hence  we immediately see that $u$ does not vanish to infinite order in the $L^2$ sense at $0$. In other words, there exists some $N>0$, such that $\varlimsup_{r\rightarrow \infty}r^N\int_{|x|>r} |u|^2 > 0.$ Thus  $\varlimsup_{r\rightarrow \infty}r^{N+1}\int_{|x|>r} |u|^2 =\infty.$

\end{proof}
\medskip

The following two examples show that the unique continuation at infinity in Theorem \ref{maini} fails  if   \eqref{e15} and/or \eqref{e14} is dropped.

\begin{example}\label{ex1}
   Given any $C^2$ function $\phi$   on $\mathbb R^+\cup\{0\}$ such that  $\phi(r)=1 $ when $0\le r\le \frac{1}{2}$ and $\phi(r) =r$ when $r\ge 1$, let $u = e^{-e^{\phi(r)}}$ on $\mathbb R^n$. Then   $u$ vanishes to infinite order in the $L^2$ sense at $\infty$ and is a weak solution to 
 $$ -\Delta u = Vu \ \ \text{on}\ \ \mathbb R^n$$
  with $$V: = -e^{2\phi(r)}\left(\phi'(r)\right)^2 + \left(\phi''(r)+\left(\phi'(r)\right)^2 +\frac{n-1}{r}\phi'(r)\right)e^{\phi(r)}.$$
  Clearly, $V\in L_{loc}^\infty(\mathbb R^n)$. 
  Note that $V= \BigO{e^{2r}}$ when $r>>1$.
\end{example}

\begin{example}\label{ex2}
 For every $\epsilon \ge 2$, the function $u = e^{- r^{ \epsilon}}$ vanishes to infinite order in the $L^2$ sense at $\infty$, and is a weak solution to 
 $$ -\Delta u = Vu \ \ \text{on}\ \ \mathbb R^n$$
  with $$V: =  -\epsilon^2r^{ 2\epsilon-2} + ((n-2)\epsilon +   \epsilon^2)r^{ \epsilon -2} \in L^\infty_{loc}(\mathbb R^n).$$ Then $V = \BigO{r^{2\epsilon-2 }}$ when $r>>1$. Note that a direct computation shows that  when $r>>1$,   $$(r^2 V)' = - r^{\epsilon-1}\epsilon \left( 2\epsilon^2 r^\epsilon -(n-2)\epsilon -   \epsilon^2)\right))< 0.$$  
\end{example}
\medskip


On the other hand, it is worth pointing out that there are   many  functions satisfying the assumptions for $V$ in Theorem \ref{maini}. Thus the theorem can be readily  applied to obtain the unique continuation property of $-\Delta u = V u  $ at infinity for such potentials.

  \begin{example}\label{ex7}
      Given any measurable nonnegative function $g$ on $\mathbb R^n$ such that $g=\BigO{|x|^M}$ for some constant $M$ when $x>>1$, let $V(r, \theta) = \frac{1}{r^2}\int_1^r g(r, \theta)dr$ for a.e.  $(r, \theta)$ in its polar coordinates. Then  $V$ satisfies all the assumptions in  Theorem \ref{maini}. 
  \end{example}

We end  the section by providing  a simple case of the Landis conjecture when the potential decays sufficiently  near infinity.

\begin{pro}\label{tr1}
    Let $V\in L^\infty(\mathbb R^n), n\ge 2$ and  $V =\BigO{ \frac{1}{|x|^M} }$ near $\infty$ for some constant $M\ge 2$. Let $u\in L_{loc}^2(\mathbb R^n)$ be a weak solution to 
  $$ -\Delta u = V u \ \ \text{on} \ \  \mathbb R^n. $$
 If  
      $\lim_{r\rightarrow \infty}r^m\int_{|x|>r} |u|^2 = 0$ for each $m\ge 0$, then $u\equiv 0$. In particular, if $V$ has compact support, then  $u\equiv 0$ whenever $u$ vanishes to infinite order in the $L^2$ sense at $\infty$. 
\end{pro}

\begin{proof}
    Applying the Kelvin transform as in Lemma \ref{be}, we have  $$ -\Delta w = Ww  \ \ \text{on}\ \ B_1\setminus\{0\}$$
    for $w$ and $W$ defined in \eqref{e16}. By Corollary \ref{HPC}, $w\in W^{2,2}_{loc}(B_1)$, vanishes to infinite order in the $L^2$ sense at 0, and satisfies $ -\Delta w = Ww $   on  $B_1$. On the other hand, when $M\ge 2$,  $W = \BigO{|x|^{-4+M}}\le \BigO{|x|^{-2}}$. According to the  unique continuation property of  \cite{Pa} for potentials of the form $C|x|^{-2}$, we see that $w= 0$ on $B_1$  and thus $u\equiv 0$.

    \end{proof}



The proof of Proposition \ref{tr1} immediately yields the following unique continuation property for harmonic functions near infinity. 

\begin{cor}\label{tr}
  Suppose $u$ is harmonic  on  $\mathbb R^n\setminus B_{r_0}$ for some $r_0>>1$ and $\lim_{r\rightarrow \infty}r^m\int_{|x|>r} |u|^2 = 0$ for each $m\ge 0$. Then $u\equiv 0$ on  $\mathbb R^n\setminus B_{r_0}$. Namely, if $u\ne 0$ somewhere on  $\mathbb R^n\setminus B_{r_0}$, then there exists some  $N\ge 0$ such that $\varlimsup_{r\rightarrow \infty}r^N\int_{|x|>r} |u|^2 = \infty$.
\end{cor}

It is  not clear to us  but would be desirable to know if  the   flatness  assumption of  $u$ at infinity in Proposition \ref{tr1} could  be relaxed to  a finite order vanishing at infinity in particular when $V$ has compact support. On the other hand, we  point  out that for the truncated domain $ \mathbb R^n\setminus B_{r_0}$ in Corollary  \ref{tr}, the $L^2$ flatness  assumption of $u$ at infinity  can not be weakened. In fact, the following simple example gives harmonic functions on $\mathbb R^n\setminus B_1$  that vanishes at infinity at any given order in the $L^2$ sense.

\begin{example}\label{ex3}
Let $P_k$ be a   homogeneous harmonic polynomial of degree $k\ge 0$ 
  on $\mathbb R^n$. Then the function $u  = \frac{1}{|x|^{n-2}}P_k\left(\frac{x}{|x|^2}\right)$   is harmonic on $\mathbb R^n\setminus B_1$. 
Note that    $u$ is only of $ \BigO{\frac{1}{|x|^{k+n-2}}}$ at $\infty$. In particular, for each  $m\ge 2k-4+n $,
$$\lim_{r\rightarrow \infty}r^m\int_{|x|>r} |u|^2 \ne 0  $$
\end{example}

Proposition \ref{tr1} can be compared with a uniqueness result of Chirka and Rosay \cite{CR} for $\bar\partial$, which  states that for $V\in L^\infty(\mathbb R^2)$ with compact support, any solution $u$ to $$\bar\partial u = V u\ \ \text{on} \ \ \mathbb R^2 $$ with $ \lim_{z\rightarrow \infty}u =0 $ must vanish identically.  Interestingly, the potential $V$ in Meshkov's  example \cite{Me}   vanishes on  a family of concentric annuli with radii going  to infinity.


   

\section{Sign changing property near infinity}
Once again, we will use the  Kelvin transform to convert the solution behavior at infinity to that near the origin. As seen in the previous sections, the order of the potential's singularity is changed as a result of   the transform, which needs be taken care of in discretion.  

\begin{theorem} \label{ca}
Let $W\in L_{loc}^\infty(B_1\setminus\{0\}), W\ge 0$  on $B_1$, 
\begin{equation}\label{e5}
  W\notin L^1 \ \ \text{and}\ \   \frac{1}{|x|^{n+2}W }\in L^1
\end{equation}
  near $0$.  Let   $w\in L_{loc}^1(B_1\setminus \{0\})$  be a weak  solution to 
\begin{equation}\label{nr}
    -\Delta w = Ww\ \ \text{on} \ \ B_1\setminus \{0\},
\end{equation}
with  \begin{equation}\label{e7}
     w = \BigO{\frac{1}{|x|^{n-2}}}
 \end{equation} near $0$. If $w\ge 0$ (or  $w\le 0 $)   on $B_1$, then $w\equiv 0$. 
\end{theorem}

We first need the following preparation lemmas to resolute the singularity at $0$.
\begin{lem}\label{l1}
  For $W$ and $w$ in Theorem \ref{ca}, one has
  $$ Ww\in L^1_{loc}(B_1).$$
\end{lem}

\begin{proof}
   Since $Ww\in L^1_{loc}(B_1\setminus\{0\})$, we only need to show that $Ww\in L^1$ near $0$. Let $\zeta_\epsilon$ be a cut-off function on $B_1$ such that $\zeta_\epsilon = 0$ when $|x|\le\epsilon$ and $|x|\ge\frac{3}{4}$; $\zeta_\epsilon = 1$ when $2\epsilon\le|x|\le\frac{1}{2}$;   $|\nabla^k \zeta_\epsilon|\lesssim  {\epsilon^{-k}}$ on $ \epsilon<|x|<2\epsilon$, and $|\nabla^k \zeta_\epsilon|\lesssim 1$ on $\frac{1}{2}<|x|<\frac{3}{4}, k\le 2$.

 Since $\zeta_\epsilon^4$ is a  a testing function   for $B_1\setminus \{0\}$, we have by \eqref{nr}
 \begin{equation}\label{e10}
     \int_{B_1} -\zeta_\epsilon^4\Delta w = \int_{B_1} Ww \zeta_\epsilon^4 .
 \end{equation}
The right hand side is nonnegative since  $Ww\ge 0$ by assumption. Using Stokes' theorem, the left hand side 
 \begin{equation}\label{e11}
     \begin{split}
     \int_{B_1} -\zeta_\epsilon^4\Delta w =     \int_{B_1} -w\Delta \zeta_\epsilon^4     =   \int_{\epsilon<|x|<2\epsilon}-w \Delta \zeta_\epsilon^4  + \int_{\frac{1}{2}<|x|<1}-w \Delta \zeta_\epsilon^4.
     \end{split}
 \end{equation}
Firstly, \begin{equation}\label{e12}
    \left|\int_{\frac{1}{2}<|x|<1}-w \Delta \zeta_\epsilon^4\right|\le  C_1
\end{equation}
for some constant $C_1>0$.
By the fact that  $  \Delta \zeta_\epsilon^4 = 4\zeta_\epsilon^2\left(3|\nabla \zeta_\epsilon|^2+ \zeta_\epsilon\Delta \zeta_\epsilon\right) $ and H\"older inequality, 
 $$ \left|\int_{\epsilon<|x|<2\epsilon}w \Delta \zeta_\epsilon^4 \right|\lesssim \int_{\epsilon<|x|<2\epsilon} {\epsilon^{-2}} w \zeta_\epsilon^2\lesssim \int_{\epsilon<|x|<2\epsilon}\frac{w \zeta_\epsilon^2}{|x|^2} \le \left(\int_{\epsilon<|x|<2\epsilon}  Ww \zeta_\epsilon^4  \right)^\frac{1}{2}\left(\int_{\epsilon<|x|<2\epsilon}\frac{w}{ |x|^4W} \right)^\frac{1}{2}. $$
Note that  \begin{equation}\label{e18}
    \frac{w}{ |x|^4W}\in L^1(B_\frac{1}{2}) 
\end{equation} by \eqref{e5} and \eqref{e7}.  Then $\left(\int_{\epsilon<|x|<2\epsilon} \frac{w}{ |x|^4W}\right)^\frac{1}{2}\le C_2$ for a constant $C_2>0$ and thus
\begin{equation}\label{e9}
    \left|\int_{\epsilon<|x|<2\epsilon}w \Delta \zeta_\epsilon^4 \right|\le C_2 \left(\int_{\epsilon<|x|<2\epsilon}  Ww \zeta_\epsilon^4  \right)^\frac{1}{2}\le C_2  \left(\int_{B_1}  Ww \zeta_\epsilon^4  \right)^\frac{1}{2}.
\end{equation} 
Combining \eqref{e10}-\eqref{e9}, we have for $A: = \left(\int_{B_1} Ww \zeta_\epsilon^4 \right)^\frac{1}{2}$,
$$  A^2\lesssim C_2A +C_1.$$
This implies $A$ is bounded by a constant dependent only on $C_1$ and $C_2$.  Passing $\epsilon$ to $0$, we have the desired integrability property for $ Ww$ near $0$. 

\end{proof}

\begin{lem}\label{rs}
  For $W$ and $w$ in Theorem \ref{ca}, we have   $w$ is a weak solution to \begin{equation*} 
    -\Delta w = Ww\ \ \text{on} \ \ B_1.
\end{equation*}
\end{lem}

\begin{proof}
 In view of   Lemma \ref{HP}, we only need to prove that 
 \begin{equation*}
      \lim_{\epsilon\rightarrow 0} \epsilon^{-2}\int_{|x|<\epsilon}|w |    =0.
 \end{equation*}
 Indeed, 
 \begin{equation*}
  \epsilon^{-2} \int_{|x|<\epsilon}|w |  \le \int_{|x|<\epsilon}  \frac{ w}{|x|^2}\le \left(\int_{|x|<\epsilon}   Ww\right)^\frac{1}{2}\left(\int_{|x|<\epsilon} \frac{w}{ |x|^4W}\right)^\frac{1}{2},
 \end{equation*}
 which goes to zero as $\epsilon \rightarrow 0$ due to  the facts that $Ww\in L^1_{loc}(B_1)$ in Lemma \ref{l1} and $ \frac{w}{ |x|^4W}\in L^1_{loc}(B_1)$   in \eqref{e18}. Note that an inspection of the proof to Lemma \ref{l1} indicates that the assumption $W\notin L^1$ in Theorem \ref{ca} is actually not needed here.
 
\end{proof}

\begin{proof}[Proof of Theorem \ref{ca}: ] Assume by contradiction that $w\not\equiv 0$. By 
  Lemma \ref{rs}, we have $-\Delta w = Ww $ on $B_1$. In particular, $-\Delta w \ge 0$ on $B_1$. Namely, $w$ is a superharmonic function a.e. on $B_1$. Consequently, $w$ is lower semi-continuous on $B_1$.  Since  $w\ge 0$ on $B_1$,    there   exists some positive number $\gamma_0$ and a small neighborhood of $0$, say, $B_\epsilon$, such that $w\ge \gamma_0 $ almost everywhere on $B_\epsilon$. See, for instance, \cite{Ho2}. Noting that  $W\ge 0$, we further have  
$  -\Delta w\ge  \gamma_0 W\ge 0 $ on $B_\epsilon$. Since $   W\notin L^1$ near $0$, 
$$\int_{B_\epsilon} - \Delta w   \ge  \gamma_0\int_{B_\epsilon} W =\infty. $$ 
However,  this would contradict with  the fact that $  \int_{B_\epsilon}-\Delta w   =  \int_{B_\epsilon}Ww,  $  which is finite according to Lemma \ref{l1}. With $w$ replaced by $-w$, one can prove the case when $w\le 0$ on $B_1$. The proof is completed.
    
\end{proof}

\begin{proof}[Proof of Theorem \ref{main2} and Corollary \ref{ho2}:]  After performing the Kelvin transform as in Lemma \ref{be}, we have by Lemma \ref{be} {\it (3)} that $W  $   defined by \eqref{e16} satisfy the assumption in Theorem \ref{ca}. The rest of the proof to Theorem \ref{main2} is then  a consequence of Theorem \ref{ca}.  

In the setting of Corollary \ref{ho2}, one can also directly verify that $V$   satisfies both assumptions in  Theorem \ref{main2} as long as $\alpha>-2$ and $\alpha\ge n-4$.  Thus Corollary \ref{ho2} follows from Theorem \ref{main2}.

\end{proof}

In particular, as a consequence of Corollary \ref{ho2},  every non-trivial bounded weak solution to  $$ -\Delta  u = \frac{c^2}{|x|^\alpha} u\ \ \text{on}\ \ \mathbb R^n$$  must change sign near infinity if  $\alpha>-2$ for $n\le 2$, or $\alpha \ge  n-4  $ for $n\ge 3$.
\medskip

As will be seen immediately, Theorem \ref{ca}  only allows us to prove Theorem \ref{ei} when the dimension $n$ is small. To cover the case of higher dimensions, we shall need the following variant of Kato's inequality: 
\begin{lem}\cite{Ka, BC}
    let $w\in L_{loc}^1(\Omega)$, $f\in L_{loc}^1(\Omega)$ and $-\Delta w \ge f$ on $\Omega$ in the distribution sense. Let $\phi: \mathbb R\rightarrow \mathbb R$ be a Lipschitz concave function with $0\le \phi'\le C$ on $\mathbb R$. Then $\psi: =\phi\circ w\in L_{loc}^1(\Omega)$ and 
$$ -\Delta \psi \ge \phi'(w)  f\ \ \text{on}\ \ \Omega$$
in the distribution sense. 
\end{lem}

\begin{proof}[Proof of Theorem \ref{ei}: ] Repeating a similar procedure as in Lemma \ref{be} with $V\equiv c^2$, it boils down to prove for $W: = \frac{c^2}{|x|^4}$ on $B_1$, every weak solution to $-\Delta w = Ww$ on $B_1\setminus \{0\}$ must change sign in every neighborhood of $0$. By replacing  $w$ by $-w$, this is further reduced to showing that   every nonnegative weak solution to $-\Delta w = Ww$ on $B_1\setminus \{0\}$ must be identically zero.

If $n\le 4$, clearly  $W\in L_{loc}^\infty(B_1\setminus\{0\}), W\ge 0, W\notin L^1 $ near $0$, and $\frac{1}{|x|^{n+2}W }\in L^1(B_1)$. Moreover $   w(x) =  \frac{1}{|x|^{n-2}}u(\frac{x}{|x|^2}) = \BigO{\frac{1}{|x|^{n-2}}}$ near $0$ by the boundedness assumption of $u$. If $w\ge 0$, then one can   apply  Theorem \ref{ca} to conclude that $w\equiv 0$.

    For $n\ge 5$,  assume by contradiction that  $w\not\equiv 0$ on $B_1$. Since $w\ge 0$ and $-\Delta w\ge 0$ on $B_1$, we have $ w\ge \gamma_0$ almost everywhere for some positive $\gamma_0$ on a small neighborhood $B_\epsilon$ of $0$.
Let $\phi: \mathbb R\rightarrow \mathbb R$ be the function such that $\phi(s) =0$ when $s\le {\gamma_0}$;     $\phi(s) = \log \frac{1}{{\gamma_0}} - \log \frac{1}{s}$ when $ s\ge {\gamma_0}$. It is straight forward to verify that $\phi$ satisfies the assumption in Kato's inequality. Let $\psi: = \phi\circ w $. Then \begin{equation}\label{e13}
    0\le \psi\le \log\frac{1}{{\gamma_0}}\ \ \text{on}\ \ B_1,
\end{equation}  and $$\psi = \log\frac{1}{{\gamma_0}} - \log \frac{1}{w} \ \ \text{on}\ \ B_\epsilon.$$ By Kato's inequality,
    $$ -\Delta \psi\ge \phi'(w) Ww = \frac{c^2}{|x|^4}\ \ \text{on}\ \ B_\epsilon.$$

    On the other hand, note that $\Delta \left(\frac{1}{|x|^{2}}\right) = \frac{ 8-2n }{ |x|^{4}}$. So
    $   -\Delta \psi \ge \Delta \left(\frac{c^2}{(8-2n) |x|^{2}}\right) $ on $B_\epsilon $, or equivalently,
     $$  -\Delta \left(  \psi-  \frac{c^2}{(2n-8) |x|^{2}}\right)\ge 0\ \ \text{on}\ \ B_\epsilon.$$
Due to the fact that $\psi-  \frac{c^2}{(2n-8) |x|^{2}}\in L^1(B_\epsilon) $,    there exists some   constant $\beta>0$ such that 
     $$ \psi \ge  \frac{c^2}{(2n-8) |x|^{2}} - \beta\ \ \text{on}\ \ B_\frac{\epsilon}{2}.$$
    However,  $2n-8>0$ when $n\ge 5$, which would imply that $\psi$ is unbounded from above, contradicting  \eqref{e13}. 
     
\end{proof}


\bigskip

\bibliographystyle{alphaspecial}

\fontsize{11}{11}\selectfont

\vspace{0.7cm}

\noindent pan@pfw.edu,

\vspace{0.2 cm}

\noindent Department of Mathematical Sciences, Purdue University Fort Wayne, Fort Wayne, IN 46805-1499, USA.\\

\noindent zhangyu@pfw.edu,

\vspace{0.2 cm}

\noindent Department of Mathematical Sciences, Purdue University Fort Wayne, Fort Wayne, IN 46805-1499, USA.\\

\end{document}